\documentclass[12pt,reqno]{article}

\usepackage{amsmath}
\usepackage{amsthm}
\usepackage{amscd}
\usepackage{amssymb}
\usepackage{graphicx}
\usepackage{latexsym}

\newcommand{\LL}{{\cal L}}
\newcommand{\CA}{{\cal A}}
\newcommand{\CB}{{\cal B}}

\newcommand{\LG}{\mathfrak L}

\newcommand{\wbar}{\overline}

\newcommand{\vbar}{\, | \,}

\def\epsilon{\varepsilon}
\def\Pr{\mathbb P}

\newcommand{\supp}{\mbox{Supp}}
\newcommand{\Curr}{\mbox{Curr}}
\newcommand{\Out}{\mbox{Out}}
\newcommand{\Aut}{\mbox{Aut}}
\newcommand{\inv}{^{-1}}

\newcommand{\ssm}{\smallsetminus}

\newcommand{\FN}{F_N}   
\newcommand{\cvn}{\mbox{cv}_N}

\newcommand{\CQ}{{\cal Q}}

\newcommand{\R}{\mathbb R}

\newcommand{\N}{\mathbb N}
\newcommand{\Hy}{\mathbb H}

\def\qed{\hfill\rlap{$\sqcup$}$\sqcap$\par}
\def\bar{\overline}
\def\tilde{\widetilde}

\newtheorem{thm}{Theorem}[section]

\newtheorem{cor}[thm]{Corollary}
\newtheorem{lem}[thm]{Lemma}
\newtheorem{prop}[thm]{Proposition}

\newtheorem{conj}[thm]{Conjecture}
\newtheorem*{acknowledgements}{Acknowledgements}

\newcommand{\affiliationone}{}
\newcommand{\email}{\\}
\renewcommand{\and}{ and }

\theoremstyle{definition}
\newtheorem{defn}[thm]{Definition}
\newtheorem{rem}[thm]{Remark}

\title{$\R$-trees and laminations for free groups III:
Currents and dual $\R$-tree metrics}

\author{Thierry Coulbois, Arnaud Hilion \and Martin Lustig}

\begin{document}

\maketitle


\section{Introduction}
\label{sec:introIII}

A geodesic lamination $\LG$ on a closed hyperbolic surface $S$, when
provided with a transverse measure $\mu$, gives rise to a ``dual
$\R$-tree'' $T_{\mu}$, together with an action of $G = \pi_{1} S$ on
$T_{\mu}$ by isometries.  A point of $T_{\mu}$ corresponds precisely
to a leaf of the lift $\tilde \LG$ of $\LG$ to the universal covering
$\tilde S$ of $S$, or to a complementary component of $\tilde \LG$ in
$\tilde S$. The $G$-action on $T$ is induced by the $G$-action on
$\tilde S$ as deck transformations.  This construction is well known
(see \cite{morg}).  It is also known \cite{skora} that conversely, for
every small isometric action of a surface group $G = \pi_{1} S$ on a
minimal $\R$-tree $T$ there exists a ``dual'' measured lamination
$(\LG, \mu)$ on $S$, i.e. one has $T = T_{\mu}$ up to a
$G$-equivariant isometry.

\smallskip

This beautiful correspondence has tempted geometers and group
theorists to investigate possible generalizations, and the first one
arises if one replaces the closed surface by a surface with boundary,
and correspondingly the surface group $G$ by a free group $\FN$ of
finite rank $N \geq 2$. A first glimpse of the potential problems can
be obtained from two simultaneous but distinct identifications $\FN
\overset{\cong}{\longrightarrow} \pi_{1} S_{1}$ and $\FN
\overset{\cong}{\longrightarrow} \pi_{1} S_{2}$, thus obtaining
actions of $\pi_{1} S_{1}$ on a tree $T_{2}$ which are dual to a
measured lamination on $S_{2}$, but in general not dual to any
measured lamination on the surface $S_{1}$.

\smallskip

Worse, using the index of an $\R$-tree action by $\FN$ as introduced
in \cite{gl}, it is easily seen that for many (perhaps even ``most'')
small or very small $\R$-trees $T$ with isometric $\FN$-action there
is no identification whatsoever of $\FN$ with the fundamental group of
any surface that would make $T$ dual to a lamination. An example of
such trees are the forward limit trees $T_{\alpha}$ of certain
irreducible automorphisms with irreducible powers (so called {\em
iwip} automorphisms) of $\FN$.  Much like pseudo-Anosov surface
homeomorphisms, such an iwip automorphism has precisely one forward
and one backward limit tree, $T_{\alpha}$ and $T_{\alpha\inv}$
respectively, and it induces a North-South dynamics on the space
$\bar{CV}_{N}$ of projectivized very small $\FN$-actions on $\R$-trees
(see \cite{ll4}).  Note that, contrary to the case of pseudo-Anosov
homeomorphisms, it is a frequent occurence for an iwip automorphism
$\alpha$ (see Corollary~\ref{cor:paranon} below) that its stretching
factor $\lambda_{\alpha}$ is different from the stretching factor
$\lambda_{\alpha\inv}$ of its inverse.

\smallskip

In \cite{chl1-II} for any $\R$-tree $T$ with isometric $\FN$-action, a
{\em dual lamination} $L(T)$ has been defined, which is
the  generalization of the geodesic lamination $\LG$ for a
surface tree $T_{\mu}$ as discussed above.  The goal of the present
paper is to investigate the effect of putting an invariant measure
$\mu$ on the dual lamination $L(T)$, or, in the proper technical
terms, considering a free group current $\mu$ with support contained
in $L(T)$. We prove here, if the $\FN$-action on $T$ is very small
and has dense orbits, that such a current defines indeed an induced
measure on the metric completion $\wbar T$ of $T$.

\smallskip

In the special case considered above where $T = T_{\mu}$ is dual to a
measured lamination $(\LG, \mu)$ on a surface, then the transverse
measure $\mu$ defines indeed a current on $L(T_{\mu})$, and the
induced measure on $\wbar T_{\mu}$ defines a dual distance on
$T_{\mu}$ which is precisely the same as the original distance on
$T_{\mu}$ (i.e. the one that comes from the transverse measure $\mu$
on $\LG$).  For arbitrary very small trees $T$ with dense
$\FN$-orbits, the measure on $\bar T$ induced by a current $\mu$
on $L(T)$ defines also a metric on $T$, except that this {\em dual metric}
$d_{\mu}$ may in general be in various ways degenerate (compare \S
\ref{subsec:dualmetric} below).  In particular, the dual distance may
well be infinite for any two distinct points of $T$.  Alternatively,
it could be zero throughout the interior $T$ of $\bar T$.

\smallskip

The main result of this paper is to show that these ``exotic''
phenomena are not just theoretically possible, but that they actually
do occur in important classes of examples.

\smallskip

Let $\alpha \in Aut(F_{N})$ be an iwip automorphism, let $T_\alpha$ be
the forward limit tree of $\alpha$. Then the dual lamination
$L(T_\alpha)$ is uniquely ergodic (see
Proposition~\ref{prop:invtreegen}): it carries a projectively unique
non-trivial current $\mu$.  In this case the dual metric $d_\mu$ is
simply called the {\em dual distance} $d_{*}$ on $T_\alpha$ or on
$\bar T_{\alpha}$.

\begin{thm}
\label{thm:throughout}
Let $\alpha \in Aut(F_{N})$ be an iwip automorphism with
$\lambda_{\alpha} \neq \lambda_{\alpha\inv}$. Then the dual distance
$d_{*}$ on the forward limit tree $T_{\alpha}$ is zero or infinite
throughout $T_{\alpha}$.
\end{thm}

Geodesic currents have been introduced by F.~Bonahon for hyperbolic
manifolds \cite{bona1,bona2}.  They turn out to be a powerful tool,
and they also admit generalizations to a much larger setting, compare
\cite{furm}.  For free groups and their automorphisms, the first
serious application was given in the thesis of M.~Bestvina's student
R~Martin \cite{mart}.  Recently, I.~Kapovich rediscovered currents and
studied them systematically, see \cite{kapo1,kapo2}.  As Kapovich's
papers are very carefully written and very accessible to non-experts,
we will review geodesic currents here only briefly and refer for all
of the basic detail work to the papers of Kapovich.

\smallskip

The novelty in the setup presented here is the relationship between
currents and laminations, which we establish systematically through
studying, for any current $\mu$, the {\em support} $\supp(\mu)$. The
latter belongs to the space $\Lambda(\FN)$ of laminations for the free
group $\FN$, which has been defined and investigated in detail in
\cite{chl1-I}. This gives a rather natural map from the space
$\Curr(\FN)$ of currents to the space $\Lambda(\FN)$.  The space
$\Curr(\FN)$ of currents $\mu$, as well as the resulting compact space
$\Pr\Curr(\FN)$ of projectivized currents $[\mu]$, admit a natural
action of the group $\Out(\FN)$ of outer automorphisms of $\FN$.  The
results derived in Proposition~\ref{currenttolamination} and in
Lemmas~\ref{lem:noninjective}, \ref{lem:pseudocontinuous} and
\ref{lem:pseudosur} can be summarized as follows:

\begin{thm}
\label{thmtwo}
The map $\supp: \Curr(\FN) \to \Lambda(\FN)$ induces a map
\[
\Pr \supp: \Pr\Curr(\FN) \to \Lambda(\FN)
\]
which has the following properties:

\begin{enumerate}
\item $\Pr \supp$ is $\Out(\FN)$-invariant.
\item $\Pr \supp$ is not injective.
\item $\Pr \supp$ is not surjective. However, every 
lamination $L \in \Lambda(\FN)$ possesses a sublamination
$L_{0} \subset L$ which belongs to the image of the map
$\Pr\supp$.
\item $\Pr \supp$ is not continuous. However, if $(\mu_{n})_{n \in
\N}$ is a sequence of currents which converges to a current $\mu$,
then the sequence of algebraic laminations $\Pr\supp([\mu_{n}])$ has
at least one accumulation point in $\Lambda(\FN)$, and any such
accumulation point is a sublamination of $\Pr\supp([\mu])$.
\end{enumerate}
\end{thm}

\medskip

Summing up, we believe that the results presented in this paper can be
interpreted as follows:

\smallskip

On the one hand, the complete correspondence between small $\R$-trees
and measured laminations, as known from the surface situation, does
not fully extend to the world of free groups, very small $\R$-tree
actions and currents. Unexpected degenerations seem to occure almost
as a rule, and much further research is needed before one can speak of
a ``true understanding''.

\smallskip

On the other hand, the spaces of currents, of $\R$-tree actions, and
of algebraic laminations for $\FN$ are naturally related, and although
this relationship is more challenging than in the surface case, there
is clearly enough interesting structure there to justify further
research efforts. A small such further contribution has already been
given, in \cite{chl2}, where algebraic laminations where used to
characterize $\R$-trees up to $\FN$-equivariant variations of their
metric.

\bigskip

\begin{acknowledgements}
This paper originates from a workshop organized at the CIRM in April
05, and it has greatly benefited from the discussions started there
and continued around the weekly Marseille seminar ``Teichm\"uller''
(partially supported by the FRUMAM).
\end{acknowledgements}

\section{Currents on $\FN$}
\label{subsec:defcurr}

Let $\CA$ be a basis of the free group $\FN$ or finite rank $N \geq
2$, and let $F(\CA)$ denote the set of finite reduced words in
$\CA^{\pm1}$, which is as usually identified with $\FN$.  A {\em
geodesic current} for a free group $\FN$ can be defined in various
ways.  In particular, there are the following three equivalent ways to
introduce them:

\smallskip
\noindent
{\bf I. Symbolic dynamist's choice:} Consider the space $\Sigma_{\CA}$
of biinfinite reduced indexed words $Z = \ldots z_{i-1} z_{i} z_{i+1}
\ldots$ in $\CA^{\pm 1}$, provided with the product topology, the
shift operator $\sigma: \Sigma_{\CA} \to \Sigma_{\CA}$, and with the
involution $Z \mapsto Z\inv$, see \cite{chl1-I}.  A geodesic current
is a non-trivial $\sigma$-invariant finite Borel measure $\mu$ on
$\Sigma_{\CA}$. We also require that $\mu$ is {\em symmetric}: the
measure is preserved by the involution of $\Sigma_{\CA}$ given by the
inversion $Z \mapsto Z\inv$.

\smallskip
\noindent
{\bf II. Geometric group theorist's choice:} Consider the space
$\partial^{2} \FN$ of pairs $(X, Y)$ of boundary points $X \neq Y \in
\partial \FN$, endowed with the ``product'' topology, with the
canonical diagonal action of $\FN$, and with the flip involution $(X,
Y) \mapsto (Y, X)$ as specified in \cite{chl1-I}.  A geodesic current
is a non-trivial $\FN$- and flip-invariant Radon measure $\mu$ on
$\partial^{2} \FN$, i.e. a Borel measure that is finite on any compact
set.

\smallskip
\noindent
{\bf III. Combinatorist's choice:} A geodesic current is given by a
non-zero function $\mu: \FN = F({\cal A}) \to \R_{\geq 0}$ with
$\mu(w\inv) = \mu(w)$ for all $w \in F(\CA)$, which satisfies the left
and the right {\em Kolmogorov property}: For all reduced words $w =
y_{1} \ldots y_{k} \in F({\cal A})$ one has 
\[\mu(w) = \sum_{y \in
{\cal A} \cup {\cal A}^{-1} \smallsetminus \{y_{k}^{-1}\}} \mu(w y) =
\sum_{y \in {\cal A} \cup {\cal A}^{-1} \smallsetminus \{y_{1}^{-1}\}}
\mu(y w).
\]

\bigskip

This three viewpoints correspond to the three equivalent definitions
given in \cite{chl1-I} of a lamination for the free group $\FN$.  We
assume some familiarity of the reader with these three settings and
will freely consider that a lamination is altogether symbolic
(viewpoint~I), algebraic (viewpoint~II) and, a laminary language
(viewpoint~III).  Whenever necessary, we emphasize the
particular viewpoint used, by notationally specifying the lamination
$L$ in question as symbolic lamination $L_{\CA}$, algebraic lamination
$L^{2}$, or as laminary language $\LL$ respectively.

For currents, the transition between the three viewpoints is also
canonical (see \cite{kapo1}), and we will freely move from one to the
other without always notifying the reader. To be specific, the
Kolmogorov value $\mu(w)$ of a reduced word $w = y_{1} \ldots y_{k}
\in F({\cal A})$, from the viewpoint~III, is precisely the measure of
the {\em cylinder}
\[ 
C_{\CA}(w) = \{\ldots z_{i-1} z_{i} z_{i+1}\ldots \mid z_{1} =
y_{1}, \ldots, z_{k} = y_{k} \} \subset \Sigma_{\CA}
\] 
from viewpoint I, and also, corresponding to viewpoint~II, equal to
the measure of the {\em algebraic cylinder} $C_{\CA}^{2}(w) \subset
\partial^{2}\FN$ given by
\[ 
\{(X, wX') \mid X =x_{1} x_{2} \ldots, X' = x'_{1} x'_{2} \ldots
\in \partial F({\cal A}), x_{1} \neq y_{1}, x'_{1} \neq y_{k}^{-1}\} \
.
\] 
Note that the algebraic cylinder $C_{\CA}^{2}(w)$ is the image of
the ``symbolic" cylinder $C_{\CA}(w)$ under the map $\Sigma_{\CA} \to
\partial^{2}\FN, \, Z = Z_{-} \cdot Z_{+} \mapsto (Z_{-}\inv, Z_{+})$.

\begin{rem}
\label{rem:cylinderbasischange}
The reader should notice that in viewpoints~I and III a basis $\cal A$
of $\FN$ is crucially used, while II is ``algebraic''. It is very
important to remember that basis change induces on the Kolmogorov
function a more complicated operation than just rewriting the given
group element in the new basis $\cal B$. The correct transition is
given, for any reduced word $w \in F({\cal B})$, by decomposing the
algebraic cylinder $C_{\CB}^{2}(w) \subset \partial^{2}\FN$ into a
finite disjoint union of translates $u_{i} C_{\CA}^{2}(v_{i})$ of
properly chosen algebraic cylinders $C_{\CA}^{2}(v_{i})$, with $u_{i},
v_{i} \in F({\cal A})$, and posing:
\[
\mu(w) = \sum_{i} \mu(v_{i}).
\]
\end{rem}

\medskip

Similarly as for laminations (see \S{}1 of \cite{chl1-I}), every
element $w$ of $\FN\smallsetminus\{ 1 \}$ (or rather, every
non-trivial conjugacy class) defines an {\em integer} current
$\mu_{w}$, given (in the language of viewpoint I) as follows: If $w =
u^m$ for the maximal exponent $m\geq 1$, then the measure $\mu_{w}(C)$
of any measurable set $C \subset \Sigma_{\CA}$ is equal to $m$ times
the number of elements of $C \cap L_{\CA}(u)$, where $L_{\CA}(u)$ is
the finite set of biinfinite words of type $\ldots v v \cdot v v
\ldots$, and $v \in F(\CA)$ is any of the cyclically reduced words
conjugated to $u$ or to $u\inv$.  Alternatively, (in the language of
viewpoint~II) the current $\mu_{w}$ is given by an $\FN$-equivariant
Dirac measure $\mu_{w}$ on $\partial^{2}\FN$, defined as follows: For
every measurable set $C^{2} \subset \partial^{2} \FN$ the value of
$\mu_{w}(C^{2})$ is given by the number of cosets $g < w > \subset
\FN$ which contain an element $v$ that satisfies $v(w^{-\infty},
w^\infty) \in C^{2}$ or $v(w^{\infty}, w^{-\infty}) \in C^{2}$.  A
third equivalent definition of $\mu_{w}$ (corresponding to
viewpoint~III) is given by a count of ``frequencies'', see
\cite{kapo2}. The noteworthy fact that $\mu_{w}$ depends only on the
element $w \in \FN$ and not on the word $w \in F(\CA)$ is obvious in
the second of these definitions, but rather puzzeling if one considers
only the first or the third.

A current is {\em rational} if it is a non-negative linear combination
of finitely many integer currents.  

\begin{rem}
\label{rem:straightforward}
The above setup of the concept of currents in its various equivalent
forms, together with the canonical identification $\FN = F({\cal A })$
for any basis ${\cal A}$ of $\FN$, provides the ideal means to see
very elegantly that many of the classical measure theoretic tools from
symbolic dynamics do not depend on the underlying combinatorics of the
chosen alphabet, but are rather algebraic in their true nature.
Determining the exact point to which ergodic theory tools can be
``algebraicized'' seems to be a worthy task but goes beyond the scope
of this paper.
\end{rem}

\section{The space $\Curr(\FN)$}
\label{subsec:spacecurr}

The set of currents on $\FN$ will be denoted $\Curr(\FN)$. It comes
naturally with several interesting structures, which we will discuss
briefly in this section.  We would like to stress that this space, as
well as its projectivization, appears to be a very interesting and
useful tool for many fundamental questions about automorphisms of free
groups, and we expect that it will play an important role in the
future developpement of this subject.

\smallskip

First, the set $\Curr(\FN)$ of currents carries the weak topology,
which for any basis $\CA$ of $\FN$ is induced by the canonical
embedding of $\Curr(\FN)$ into the vector space $\R^{F({\cal A})}$,
given by $\mu \mapsto (\mu(w))_{w \in F(\CA)}$.  In particular, a
family of currents $\mu_{i}$ converges towards a current $\mu \in
\Curr(\FN)$ if and only if $\mu_{i}(w)$ converges to $\mu(w)$ for
every $w \in F({\cal A})$.

\smallskip

Next, the same formalism as explained in
Remark~\ref{rem:cylinderbasischange} for a basis change defines
canonically an action by homeomorphisms of $\Out(\FN)$ on the space
$\Curr(\FN)$, which is formally given, for any $\alpha \in \Aut(\FN)$
and any $\mu \in \Curr(\FN)$, by $\alpha_{*}(\mu) (C) = \mu(\alpha\inv
(C))$, for every measurable set $C \subset \partial^{2}\FN$.  This
convention defines a left action of $\Out(\FN)$:
\[
\alpha_{*}(\beta_{*}(\mu)) (C) = \beta_{*}(\mu)(\alpha\inv (C))
= \mu(\beta\inv(\alpha\inv (C)))
\]
\[ 
= \mu((\alpha \beta)\inv (C)) = (\alpha \beta)_{*}(\mu)(C)
\] 
For any integer current $\mu_{w}$, with $w \in \FN\smallsetminus \{1
\}$, this gives (compare \cite{kapo2,kapo1}):
\[
\alpha_{*}(\mu_{w}) = \mu_{\alpha(w)}
\]

\smallskip

Every current $\mu$ defines naturally a lamination $L(\mu)$ for the
free group $\FN$.  $L(\mu)$ can be viewed as an algebraic lamination
$L^{2}(\mu)$, i.e. a non-empty subset of $\partial^{2} \FN$ which is
closed and invariant under the $\FN$-action and the flip involution,
compare \cite{chl1-I}.  In this setting, $L^2(\mu) \subset
\partial^{2}\FN$ is simply the support $\supp (\mu)$ of the Borel
measure $\mu$ on $\partial^{2}\FN$, i.e. the complement of the biggest
open set (= the union of all open sets) with measure 0.
Alternatively, $L(\mu)$ is given via its laminary language $\LL(\mu) =
\{ w \in F({\cal A}) \mid \mu(w) > 0 \}$.  We refer to a current $\mu$
with support contained in a lamination $L$ simply as {\em an invariant
measure} on $L$. Alternatively, one says that $\mu$ is {\em carried
by} $L$.

\smallskip

A lamination $L$ which has, up to scalar multiples, only one current
$\mu$ with support $L(\mu) = L$, is called {\em uniquely ergodic}. The
simplest examples of non-uniquely ergodic laminations are given by the
union $L$ of two disjoint laminations $L_{0}$ and $L_{1}$, such that
$L_{0}$ and $L_{1}$ are the support of currents $\mu_{0}$ and
$\mu_{1}$ respectively (for example rational laminations $L_{0} =
L(a)$ and $L_{1} = L(b)$ for distinct basis elements $a, b \in \CA$).
For $0 < \lambda < 1$ one obtains an interval of pairwise projectively
distinct currents
\[ 
\mu(\lambda) =\lambda \mu_{1} + (1- \lambda) \mu_{0}\, ,
\]
all with support $L$.

\begin{prop}
\label{currenttolamination}
Recall that we assume $N \geq 2$, and let $\Lambda(\FN)$ denote the
space of laminations for $\FN$ as introduced in \cite{chl1-I}. The map
\[
\supp :  \Curr(\FN) \to \Lambda(\FN), \, \,
\mu \mapsto L(\mu)
\]
is $\Out(\FN)$-equivariant, but not continuous 
and not surjective.
\end{prop}

\begin{proof}
The $\Out(\FN)$-equivariance is a direct consequence of the definition
of the action of $\Out(\FN)$ on $\Curr(\FN)$ and on $\Lambda(\FN)$.

To see that the map $\supp$ is non-surjective it suffices to consider
the symbolic lamination $L = L_{\{a, b\}}(Z)$ generated by the
biinfinite word $Z = \ldots a a a b \cdot a a a \ldots$.  It consists
of the $\sigma$-orbit of $Z$ and of the periodic word $\ldots a a
\cdot a a \ldots$, as well as of their inverses.  However, it is an
easy exercise to show that any Kolmogorov function $\mu$ on the
associated laminary language ${\cal L}_{\{a, b\}}(Z)$, as it takes on
values in $\R_{\geq 0}$ and not in $\R_{\geq 0} \cup \{\infty \}$,
must associate the value $0$ to any word that contains the letter $b$,
so that all the measure of $\mu$ will be concentrated on the
sublamination $L_{\{a, b\}}({a})$ of $L$.

The fact that the map $\supp$ is non-continuous can be seen from the
above defined family $\mu(\lambda)$ of currents with constant support
$L$, by letting the parameter $\lambda$ converge inside the open
interval $(0, 1)$ to the value 0 (or 1): For any such $\lambda$ the
support of $\mu(\lambda)$ is clearly the union $L_{0} \cup L_{1}$,
while for the limit one gets $L(\mu(0)) = L_{0}$ (or $L(\mu(1)) =
L_{1}$).  \qquad ${}$
\end{proof}

The space $\Curr(\FN)$ has some additional structures which are not
matched by corresponding structures in $\Lambda(\FN)$. For
example, there is a canonical linear structure on $\Curr(\FN)$, given
simply by the embedding of $\Curr(\FN)$ into the real vector space
$\R^{F({\cal A})}$.  Projectivization $\mu \mapsto [\mu]$ defines the
space of {\em projectivized currents} $\Pr \Curr(\FN)$. Both
$\Curr(\FN)$ and its projectivization are infinite dimensional, but
$\Pr \Curr(\FN)$ is compact. Clearly, the map $\supp$ splits over the
projectivization, thus inducing a map $\Pr\supp: \Pr \Curr(\FN) \to
\Lambda(\FN)$, which by Proposition~\ref{currenttolamination} is
$\Out(\FN)$-equivariant, non-continuous, and non-surjective.  We
obtain furthermore

\begin{lem}
\label{lem:noninjective}
The map $\Pr\supp: \Pr \Curr(\FN) \to \Lambda(\FN)$ is
non-injective.
\end{lem}

\begin{proof}
Any non-uniquely ergodic lamination, in particular the above defined
family $\mu(\lambda)$ of currents with constant support $L$, shows
that the map $\Pr\supp$ is not injective.
\end{proof}

A second interesting example for the non-continuity of the map
$\supp$, other than the one given in the proof of
Proposition~\ref{currenttolamination}, is given by the rational
currents $\frac{1}{n}\mu_{a b^n}$ which converge to $\mu_{b}$, while
their support $L(a b^n)$ converge to the lamination generated by
$\ldots b b a \cdot b b \ldots$ and $\ldots b b \cdot b b \ldots$,
which is strictly larger than the lamination $L(b)$.

\smallskip

This last example, as also the one given in the proof of
Proposition~\ref{currenttolamination}, indicates that a weaker
statement than the continuity might be true for the map $\supp$.
Since this will be needed in \S\ref{subsec:dualmetric} as an important
ingredient for the proof of Proposition~\ref{prop:invtreegen}, we
formalize it here:

\medskip

We say that a subset $\delta$ of $\Lambda(\FN)$ is {\em saturated}
if $\delta$ contains with any lamination also 
all of its sublaminations.

\begin{lem}
\label{lem:pseudocontinuous}
Let $\delta \subset \Lambda(\FN)$ be a closed saturated subset of
laminations.  Then the full preimage $\Delta \subset
\Pr\Curr(\FN)$ of $\delta$ under the map $\Pr\supp$ is closed.
\end{lem}

\begin{proof}
We consider a sequence of currents $\mu_{k}$ in $\Curr(\FN)$, with
$L(\mu_{k}) \in \delta$ for any $\mu_{k}$. By the compactness of
$\Pr\Curr(\FN)$ and of $\Lambda(\FN)$ we can assume, after possibly
passing over to a subsequence, that there is a current $\mu \in
\Curr(\FN)$ and a lamination $L \in \Lambda(\FN)$ with $[\mu] =
\underset{k \to \infty}{\lim} [\mu_{k}]$ and $L = \underset{k \to
\infty}{\lim} L(\mu_{k})$. By properly normalizing the $\mu_{k}$ we
can actually assume that $\mu = \underset{k \to \infty}{\lim}
\mu_{k}$.

We now fix a basis $\CA$ of $\FN$ and consider the value of the
Kolmogorov function $\mu(w)$ for any $w \in \FN \ssm \{1\}$.  If
$\mu(w) > 0$, then by the topology on $\Curr(\FN)$, for any $\epsilon$
with $\mu(w) > \varepsilon > 0$ there is a bound $k_{0}$ such that for
any $k \geq k_{0}$ one has $\vbar \mu_{k}(w) - \mu(w) \vbar <
\varepsilon$.  This shows for all $k \geq k_{0}$ that $w$ belongs to
the laminary language $\LL(\mu_{k})$.  But this implies that $w$
belongs to the laminary language of $L$, which shows that $\mu$ is
carried by $L$. Since by hypothesis $\delta$ is closed and saturated,
this shows that $[\mu]$ is contained in $\Delta$, so that the latter
must be closed.
\end{proof}

\medskip

A weaker statement than the surjectivity of the map $\supp$ is
crucially used in \S\ref{subsec:dualmetric}, again in the proof of
Proposition~\ref{prop:invtreegen}:

\begin{lem}
\label{lem:pseudosur}
Every lamination $L \in \Lambda(\FN)$ contains a sublamination which
is the support of some current $\mu \in \Curr(\FN)$.
\end{lem}

\begin{proof}
For some basis $\CA$ of $\FN$, let $Z = \ldots z_{i-1} z_{i} z_{i+1}
\ldots$ be a leaf of the lamination $L$. Let $Z_{n} = z_{-n} \ldots
z_{n}$ be the central subword of $Z$ of length $2n + 1$.

For every $n \in \N$ we define a ``counting function'' $m_{n}: F(\CA)
\to \R_{\geq 0}$, by setting, for any word $w$ in $F(\CA)$, $m_{n}(w)$
to be the number of occurences of $w$ as subword of $Z_{n}$ or of
$Z_{n}\inv$, divided by $4n+2$.  It follows directly that $m_{n}$
satisfies the equations that defines the right and the left Kolmogorov
property, up to possibly an error of absolute value less than
$\frac{1}{2n+1}$. The total value of $m_{n}$ on the set of words of
length $1$ is $1$, for any $n \in \N$. Moreover $m_n(w)$ is non-zero
only for subwords of $Z$.

For each word $w$ in $F(\CA)$ we can chose a subsequence of
$(m_n)_{n\in\N}$ whose value at $w$ converges. By a diagonal argument
we get a subsequence that converges pointwise to a limit function
$\mu$ which satisfies the Kolmogorov laws while still having total
value $1$ on set of words of length $1$, so that it is non-zero.

By construction, we have $m_{n}(w) = m_{n}(w\inv)$ for all $w \in
F(\CA)$, so that the same is true for $\mu$. Hence $\mu$ is a current.
Its support is contained in the set of subwords of $Z$ and thus, as a
lamination, in $L$.
\end{proof}

A very interesting subspace ${\cal M} \subset \Pr \Curr(\FN)$ has been
introduced by R.~Martin in \cite{mart} as closure of the
$\Out(\FN)$-orbit of $[\mu_{a}]$, for any element $a$ of any basis
${\cal A}$ of $\FN$. R. Martin shows that a projectivized integer
current $[\mu_{w}]$ belongs to ${\cal M}$ if and only if $w$ is
contained in a proper free factor of $\FN$.  In contrast to the
analogous situation for $\overline{\Out(\FN) L(a)}$ (compare
Proposition~8.1 of \cite{chl1-I}), for $N \geq 3$ it has been shown in
\cite{kl}, Theorem B, that $\cal M$ is the unique minimal subspace of
$\Curr(\FN)$ which is non-empty, closed and $\Out(\FN)$-invariant.

\smallskip

The fact that currents behave somehow more friendly than laminations
is underlined by the following fact, proved in R.~Martin's thesis and
attributed there to M.~Bestvina (compare to Proposition~6.5 of 
\cite{chl1-I}):

\begin{prop}[\cite{mart}]
\label{prop:rationalcurr}
The set of projectivized integer currents $[\mu_{w}]$, for any $w
\in \FN$, is dense in
$\Pr\Curr(\FN)$.
\end{prop}


\section{Geometric currents}
\label{subsec:geomcurr}

A large class of very natural examples for a current $\mu \in
\Curr(\FN)$ is given by any geodesic lamination ${\LG} \subset S$,
provided with a transverse measure $\mu'$, where $S$ is a hyperbolic
surface with boundary as considered in the section 3 of \cite{chl1-I}
and section~6 of \cite{chl1-II}.  In this case the measure $\mu$ on
$\partial^{2} \FN$ can be nicely seen geometrically through the
canonical identification of $\partial \FN$ with the space $\partial
\tilde S$ of ends of the universal covering $\tilde S$, which is
embedded as subset in the boundary at infinity $S_\infty^1 = \partial
\Hy^2$.  Two disjoint intervals $A, B \subset S_\infty^1$, with
intersections $A' = A \cap \partial \tilde S, B' = B \cap \partial
\tilde S$, define a measurable set $A' \times B'$ of $\partial^{2}
\FN$, and the measure $\mu(A'\times B')$ is precisely given by the
measure $\mu'(\beta)$ of an arc $\beta$ in $S$ which is transverse to
${\LG}$, and which lifts to an arc $\tilde \beta$ in $\tilde S \subset
\Hy^2$ that has its two endpoints on the two extremal leaves of
${\tilde \LG} \subset\tilde S$ which bound the set of all leaves of
$\tilde \LG$ that have one endpoint in $A$ and one endpoint in $B$.


\section{The dual metric for $\R$-trees}
\label{subsec:dualmetric}

In this section we assume familiarity of the reader with the notions
of \cite{chl1-II}, from which we also import the notation without
further explanations.

In the last section we have seen that every transverse measure $\mu$
on a geodesic lamination $\LG$ which is contained in a hyperbolic
surface $S$, with non-empty boundary and with an identification
$\pi_{1} S = \FN$, gives rise to a canonical current in $\Curr(\FN)$
which we also denote by $\mu$. In section~6 of \cite{chl1-II} we have
discussed that $({\LG},\mu)$ determines an $\R$-tree $T_{\mu}$ with
isometric $\FN$-action, and that the support of the current $\mu$ and
the dual lamination of $T_{\mu}$ are the same: this lamination is
precisely the lamination associated to ${\LG} \subset S$.

One of the most intriguing aspects of the relationship between
currents and $\R$-trees comes from the attempt to extend this
correspondence, which for surfaces is almost tautological, to more
general $\R$-trees $T$.  Indeed, the goal of this section is to
understand better the true nature of the interaction between the
metric on $T$ and an invariant measure $\mu$ carried by
the dual  lamination $L(T)$ as defined in \cite{chl1-II}.

\bigskip

In the sequel we consider the dual lamination $L(T)$ as algebraic
lamination $L^{2}(T)$, i.e. a non-empty, $\FN$-invariant,
flip-invariant and closed subset of $\partial^2\FN$.  From
\cite{chl1-II} we know that there is a map ${\cal Q}^{2}: L^{2}(T) \to
\bar T$ which is $\FN$-equivariant and continuous (see Proposition~8.3
of \cite{chl1-II}).  Here $T$ is an element of the boundary $\partial
\cvn$ of the unprojectivized Outer space $\cvn \,$: in particular, $T$
is a non-trivial $\R$-tree with minimal, very small $\FN$-action by
isometries (see \cite{chl1-II}, \S 2).  We also require that the
$\FN$-orbits of points are dense in $T$ (``$T$ has dense orbits''),
and we denote by $\bar T$ the metric completion of $T$.

\begin{cor}
\label{Q2measurable}
For all $T \in \partial \cvn$ with dense orbits, the map ${\cal
Q}^{2}: L^{2}(T) \to \bar T$ is measurable (with respect to the two
Borel $\sigma$-algebras on $L^{2}(T)$ and on $\bar T$).  \qed
\end{cor}

We apply the last corollary in order to define an {\em extended
pseudo-metric} $d_{\mu}$ on $\bar T$, for any current $\mu$ which is
carried by $L(T)$. An extended pseudo-metric is just like a
metric, except that distinct points $P, Q$ may have distance 0,
positive distance, or distance $\infty$.

\begin{defn}
\label{extendedpseudodistance}
Let $T \in \partial {cv}_{N}$ be with dense orbits, and assume that
$\mu \in \Curr(F_{N})$ satisfies $\supp(\mu) \subset L(T)$. One then
defines, for any $P, Q \in \bar T$, their {\em $\mu$-distance} as
follows: 
\[
d_{\mu}(P, Q) = \mu(({\cal Q}^{2})^{-1}([P, Q]))
\, \, \, [\, = {\cal Q}^{2}_{*}(\mu)([P, Q]) \, ]
\]
\end{defn}

Clearly the function $d_{\mu}$ is symmetric and, since $\bar T$ is a
tree, it satisfies the triangular inequality. For three points $P, Q,
R \in \bar T$ with $Q \in [P, R]$ one has $d_{\mu}(P, R) = d_{\mu}(P,
Q) + d_{\mu}(Q, R)$ unless $\mu(({\cal Q}^{2})^{-1}(\{Q\})) > 0$,
which of course can happen (for example if $Q$ has non-trivial
stabilizer which carries all of the support of $\mu$).

\medskip

We distinguish now three special cases (note that we always assume
that $T$ is a minimal $\R$-tree, so that it agrees with its interior):
The metric $d_{\mu}$ is called {\em zero throughout $T$} if any two
points in $T$ have $\mu$-distance 0.  It is called {\em infinite
throughout $T$} if any two distinct points in $T$ have $\mu$-distance
$\infty$.  It is called {\em positive throughout $T$} if any two
distinct points in $T$ have positive finite $\mu$-distance.  Otherwise
we call the $\mu$-distance {\em mixed}.

\medskip

A particular case, which is of special importance, is the following:

\begin{defn}
\label{genuinetrees}
An $\R$-tree $T \in \partial cv_{N}$ is called {\em dually uniquely ergodic} if the
dual lamination $L(T)$ is uniquely ergodic.
\end{defn}

We note that, in the case where $T$ is dually uniquely ergodic, the $\mu$-distance is
uniquely determined by $T$, up to rescaling. In this case we suppress
the measure $\mu$ and speak simply of the {\em dual distance} $d_{*}$
on $T$.

\begin{conj}
\label{nomixed}
If $T$ is dually uniquely ergodic then the dual distance is not mixed.
\end{conj}

We finish this article by proving that the case of dual distances which
are infinite or zero throughout the interior does actually exist, and
that it occurs in a natural context.  We assume from now on a certain
familiarity with some of the modern tools for the geometric theory of
automorphisms of free groups.  Background material and references can
be found in \cite{vogt}.  In particular we will use below the
following facts and definitions:

\begin{rem}
\label{rem:facts}
(1) An automorphism $\alpha$ of $\FN$ is called {\em irreducible with
  irreducible powers (iwip)} if no non-trivial proper free factor of
  $\FN$ is mapped by any positive power of $\alpha$ to a conjugate of
  itself.

\smallskip
\noindent
(2) It is known (compare \cite{ll4}) that for every iwip automorphism
$\alpha$ there is, up to $\FN$-equivariant homothety, precisely one
minimal {\em forward limit} $\R$-tree $T_{\alpha} \in \partial cv_{N}$
which admits a homothety $H: T_{\alpha} \to T_{\alpha}$ with
stretching factor $\lambda_{\alpha} > 1$ that twistedly commutes with
$\alpha$.  By this we mean that
\[
\alpha(w) H = H w: T_{\alpha} \to T_{\alpha}
\]
holds for every $w \in F_{N}$.
Note that both, the map $H$ as well as the $\FN$-action on $T$, extend 
canonically to the metric completion $\bar T_{\alpha}$, so that  
the last statement holds also for $\bar T_{\alpha}$ instead of 
$T_{\alpha}$.

\smallskip
\noindent
(3) In terms of the induced action of $Out(\FN)$ on the
non-projectivized closed Outer space $\overline{cv}_{N}$ (see
\cite{chl1-II}, \S{}9), the equation in (2) can be expressed by stating
\[
T \alpha_{*} = \alpha^{-1}_{*} T = \lambda_{\alpha} T \, ,
\]
where $\lambda_{\alpha} T$ denotes the tree $T$ rescaled by the factor
$\lambda_{\alpha}$.

\smallskip
\noindent
(4) As a consequence of the equation in (2), the homothety $H$
satisfies:
\[
H \CQ^{2} = \CQ^{2} \alpha:  L^{2}(T_{\alpha})
\to \wbar T_{\alpha}\, .
\]

\smallskip
\noindent
(5) There is no further fixed point of the $\alpha_{*}$-action on 
$\overline{CV}_{N}$ other than the points $[T_{\alpha}]$ and $[T_{\alpha\inv}]$ 
specified above.  In \cite{ll4} it is shown that any iwip 
automorphism has North-South dynamics on $\overline{CV}_{N}$.

\smallskip
\noindent
(6) One knows from \cite{mart}, Theorem~30 (again attributed to
M.~Bestvina) that, if $\alpha$ is not {\em geometric}, i.e. induced by
a surface homeomorphism $h: S \overset{\approx}{\to} S$ via some
identification $\FN \cong \pi_{1} S$, then the $\alpha_{*}$-action on
$\Pr\Curr(\FN)$ possesses precisely two fixed points, an attractive
and a repelling one, and that $\alpha_{*}$ has a North-South dynamics
on $\Pr\Curr(\FN)$.

\smallskip
\noindent
(7) Let us denote by $\mu_{\alpha} \in \Curr(\FN)$ a representative of
the attracting fixed point of the $\alpha_{*}$-action on
$\Pr\Curr(\FN)$.  It satisfies $\alpha_{*}(\mu_{\alpha}) =
\lambda_{\alpha} \mu_{\alpha}$, see \cite{mart}, where
$\lambda_{\alpha}$ is the stretching factor given in (2).

\smallskip
\noindent
(8) Following \cite{mart}, the support of $\mu_{\alpha}$ is contained
in the so called {\em legal lamination} $L_{\alpha} \in \Lambda(\FN)$:
Its leaves are represented, for any train track representative $f:
\tau \to \tau$ of $\alpha$, by biinfinite legal paths in $\tau$, and
consequently by non-trivial (in fact: biinfinite) geodesics in
$T_{\alpha}$ (compare with the {\em attractive lamination} defined in
\cite{bfh}).  In particular, it follows from the alternative
definition of the dual lamination, $L(T) =L_{\CQ}(T)$, given in
Theorem~1.1 of \cite{chl1-II}, that the two laminations $L_{\alpha}$
and $L(T_{\alpha})$ are disjoint.

\smallskip
\noindent
(9) Any iwip automorphism possesses a train track representative $f:
\tau \to \tau$ with transition matrix that is primitive.  As a
consequence, any edge $e$ of $\tau$ will have an iterate $f^k(e)$
which crosses over all other edges.  The canonical image in
$T_{\alpha}$ (under the map $i: \tilde \tau \to T_{\alpha}$, see
\cite{ll4}) of any lift of $f^k(e)$ to the universal covering $\tilde
\tau$ is a segment which has the property that the union of its
$\FN$-translates covers all of $T_{\alpha}$.
\end{rem}

\begin{prop}
\label{prop:invtreegen}
For every non-geometric iwip automorphisms $\alpha \in \Aut (\FN)$,
the forward limit tree $T_{\alpha}$ is 
dually uniquely ergodic.
\end{prop}

\begin{proof}
From the $\Out(\FN)$-equivariance of the map $\lambda^{2}: \partial
{cv}_{N} \to \Lambda(\FN)$ in Proposition 9.1 of \cite{chl1-II},
together with Remark~\ref{rem:facts}~(3) above, it follows that the
dual lamination $L(T_{\alpha})$ is fixed by $\alpha$.  Hence the set
$\Delta(\alpha) \subset \Pr\Curr (\FN)$, which consists of all
preimages under the map $\Pr\supp$ of the lamination $L(T_{\alpha})$
and any of its sublaminations, is invariant under the action of
$\alpha_{*}$ (by the equivariance of the maps $\supp$ and $\Pr\supp$,
see Proposition~\ref{currenttolamination}).  As the set of all
sublaminations of a given lamination is closed, see Proposition~6.4 of
\cite{chl1-I}, it follows from Lemma~\ref{lem:pseudocontinuous} that
$\Delta(\alpha)$ is closed.  Furthermore $\Delta(\alpha)$ is
non-empty, by Lemma~\ref{lem:pseudosur}.  Thus $\Delta(\alpha)$ is the
non-empty union of closures of $\alpha_{*}$-orbits, so that it must
contain the closure of at least one $\alpha_{*}$-orbit in $\Pr\Curr
(\FN)$. From the North-South dynamics of the $\alpha_{*}$-action on
$\Pr\Curr (\FN)$ (Remark~\ref{rem:facts}~(6)) it follows that either
$\Delta(\alpha)$ consists of precisely one of the two fixed points
$[\mu_{\alpha}]$ or $[\mu_{\alpha\inv}]$, or else it contains both of
them.

But according to Remark~\ref{rem:facts}~(8) the support of
$\mu_{\alpha}$ is contained in the legal lamination $L_{\alpha}$,
which in turn is disjoint from $L(T_{\alpha})$.  Hence
$[\mu_{\alpha}]$ is not contained in $\Delta(\alpha)$, which proves
that the latter consists precisely of the point $[\mu_{\alpha\inv}]$.
This shows that $L(T_{\alpha})$ supports only one (projectivized)
current, namely $[\mu_{\alpha\inv}]$.
\end{proof}

We can now give the proof of our main result as stated in \S 1:

\begin{proof}[Proof of Theorem \ref{thm:throughout}]
From Proposition~\ref{prop:invtreegen} and its proof we know that the
forward limit tree $T_{\alpha}$ has dual lamination $L(T_{\alpha})$
which carries an (up to homothety) unique current, and that this
current is equal to $\mu_{\alpha\inv}$.

We now calculate, for any $P, Q \in T_{\alpha}$ (using
Remark~\ref{rem:facts}~(4) to get the third, and (7) to get the sixth
of the equalities below):
\[
d(H(P), H(Q)) = \lambda_{\alpha}d(P, Q)
\]
and
\[
\begin{array}[t]{rcl}d_{*}(H(P), H(Q)) & = &
\mu_{\alpha\inv}((\CQ^{2})\inv([H(P), H(Q)])) \\
&=&\mu_{\alpha\inv}((H \CQ^{2} \alpha\inv)\inv([H(P), H(Q)])) \\
&=&\mu_{\alpha\inv}(\alpha((\CQ^{2})\inv([P, Q]))) \\
&=&\alpha_{*}\inv(\mu_{\alpha\inv})((\CQ^{2})\inv([P, Q]))\\
&=&\lambda_{\alpha\inv}\mu_{\alpha\inv}((\CQ^{2})\inv([P, Q])) \\
&=&\lambda_{\alpha\inv}d_{*}(P, Q)
\end{array}
\] Assume now that some points $P \neq Q \in T_{\alpha}$ have finite
dual distance.  By iterating $H$ one finds an interval $[H^{n}(P),
H^n(Q)]$ with the property that the union of its $\FN$-translates
covers all of $T_{\alpha}$ (compare Remark~\ref{rem:facts}~(9)).  This
implies that any two points in $T_\alpha$ have finite dual distance.
If the dual distance function is furthermore non-zero, by the same
argument it follows that any two points have non-zero distance. Thus
the dual metric $d_{*}$ on $T_{\alpha}$ defines a non-trivial
$\R$-tree $T^{*}_{\alpha}$ with free $\FN$-action, and hence, since
the equation in Remark~\ref{rem:facts}~(2) carries over from
$T_{\alpha}$ to $T^{*}_{\alpha}$, the $\R$-tree $T^{*}_{\alpha}$
defines a fixed point $[T^{*}_{\alpha}]$ of the $\alpha_{*}$-action on
$\partial CV_{N}$ (see \S{}2 and \S{}9 of \cite{chl1-II}).  By
Remark~\ref{rem:facts}~(5) the point $[T^{*}_{\alpha}]$ must agree
with either $[T_{\alpha}]$ or $[T_{\alpha\inv}]$.  But this cannot be
because we computed above that the streching factor of the
$\alpha_{*}$-action on $T^{*}_{\alpha}$ is equal to
$\lambda_{\alpha\inv}$ and hence bigger than 1 (which rules out
$[T^{*}_{\alpha}] = [T_{\alpha\inv}]$), but different from
$\lambda_{\alpha}$ (thus ruling out $[T^{*}_{\alpha}] = [T_{\alpha}]$,
by Remark~\ref{rem:facts}~(3)).

Hence the dual metric $d_{*}$ must be either zero or infinite
throughout $T_{\alpha}$.
\end{proof}

A concrete example of an automorphism that satisfies the properties
stated in Theorem~\ref{thm:throughout} as hypotheses is given in
\cite{abhs} by the automorphism
\[
\begin{array}[t]{rcl} a & \mapsto & ab \\
b & \mapsto & ac \\
c & \mapsto & a
\end{array}
\] 
of $F_{3}$, which has streching factor $1, 84\ldots$, while its
inverse
\[
\begin{array}[t]{rcl} a & \mapsto & c \\
b & \mapsto & c\inv a \\
c & \mapsto & c\inv b
\end{array}
\] 
has streching factor $1, 39\ldots$.

\smallskip

An iwip automorphism $\alpha \in \Aut(\FN)$ is called {\em
parageometric}, if $\alpha$ is not geometric, but $T_{\alpha}$ is a
geometric tree (see \cite{gl, gjll}).  It has been proved recently in
\cite{hm}, see also \cite{guir}, that in this case the iwip
automorphism $\alpha\inv$ is not parageometric, and that its
stretching factor $\lambda_{\alpha\inv}$ is strictly smaller than
$\lambda_{\alpha}$ (compare \cite{gaut}).  A family of such
automorphisms, one for any $N \geq 3$, has been exhibited and
investigated in\cite{jl}.  We summarize:

\begin{cor}
\label{cor:paranon}
The dual metric on the forward limit tree of any parageometric iwip
automorphism of $\FN$, or of its inverse, is always infinite or zero
throughout.
\end{cor}

\affiliationone{
Thierry Coulbois, Arnaud Hilion and Martin Lustig\\
Math\'ematiques (LATP)\\
Universit\'e Paul C\'ezanne -- Aix-Marseille III\\
av. escadrille Normandie-Ni\'emen\\
13397 Marseille 20\\ 
France
\email{Thierry.Coulbois@univ-cezanne.fr\\
Arnaud.Hilion@univ-cezanne.fr\\
Martin.Lustig@univ-cezanne.fr\\
}
}

\end{document}